\documentclass{amsart}
\usepackage[pagewise]{lineno}
\usepackage{graphicx}
\usepackage{stmaryrd}
\usepackage{txfonts}
\usepackage{color}
\usepackage[dvipsnames]{xcolor}

\usepackage[T1]{fontenc}
\usepackage{lmodern}

\usepackage[colorlinks=true,
            linkcolor=violet,
            urlcolor=Aquamarine,
            citecolor=YellowOrange]{hyperref}

\newtheorem{thm}{Theorem}[section]

\newtheorem{lem}[thm]{Lemma}

\newtheorem{defn}[thm]{Definition}
\newtheorem{rem}[thm]{Remark}

\newcommand{\norm}[1]{\left\Vert#1\right\Vert}
\newcommand{\abs}[1]{\left\vert#1\right\vert}
\newcommand{\set}[1]{\left\{#1\right\}}
\newcommand{\Real}{\mathbb R}

\newcommand{\pfrac}[2]{\frac{\partial #1}{\partial #2}}

\definecolor{darkgreen}{rgb}{0,0.5,0}
\definecolor{darkred}{rgb}{0.7,0,0}

\title[Uniqueness of tangent cone]{
\bf
Uniqueness of tangent cone for biharmonic map\\
with isolated singularity
}
\author{Youmin Chen}
\address{School of Mathematical Sciences, University of Science and Technology of China}%
\email{chym189@mail.ustc.edu.cn}%
\author{Hao Yin}
\address{School of Mathematical Sciences, University of Science and Technology of China}%
\email{haoyin@ustc.edu.cn}%
\thanks{The research work of Hao Yin is supported by NSFC 11471300.}%

\date{\today~(@\the\time mpm)}

\begin{document}
\maketitle

\begin{abstract}
In this paper, we study the problem of uniqueness of tangent cone for minimizing extrinsic biharmonic maps. Following the celebrated result of Simon, we prove that if the target manifold is a compact analytic submanifold in $\Real^p$ and if there is one tangent map whose singularity set consists of the origin only, then this tangent map is unique.
\end{abstract}

\section{Introduction}
In this paper, we prove the biharmonic map version of the celebrated result of Simon \cite{simon1983asymptotics}. Here we restrict ourselves to the case of extrinsic biharmonic maps. Let $B\subset \Real^m$ be the unit ball around the origin and $N$ be a closed Riemannian manifold isometrically embedded in $\Real^p$. A map $u\in W^{2,2}(B,\Real^p)$ into $N$ is called (extrinsic) biharmonic map if and only if  it is the critical point of the energy
\begin{equation}
	E(u)=\int_B \abs{\triangle u}^2 dx.
	\label{eqn:energy}
\end{equation}
It is called a minimizing (biharmonic) map if for any $B_r(x)\subset B$ and $W^{2,2}$ maps $v$ with $v\equiv u$ on $B\setminus B_r(x)$, we have
\begin{equation*}
	E(v)\geq E(u).
\end{equation*}

Since the pioneering work of Chang, Wang and Yang \cite{chang1999regularity}, many authors studied the regularity problem of biharmonic maps, see \cite{strzelecki2003,wang2004biharmonic,wang2004remarks,wang2004stationary,hong2005,lamm2008}. Roughly speaking, stationary biharmonic maps are regular away from a singularity set of co-dimension $4$. For minimizing maps, one expects better regularity since it was proved by Schoen and Uhlenbeck \cite{schoen1982regularity} that minimizing harmonic maps are regular away from a singularity of co-dimension $3$. Moreover, Luckhaus \cite{luckhaus1988partial} proved the compactness of minimizing harmonic maps using a lemma which was later named after him. This compactness is crucial to the theory of singularity set of minimizing harmonic maps. We refer the readers to the book of Simon \cite{simon1996theorems} for a nice presentation of this deep theory.  The limit of a sequence of minimizing biharmonic maps was studied by Scheven in \cite{scheven2008dimension}. Instead of proving the direct analogue of Luckhaus lemma, the author studied the defect measure after Lin \cite{lin1999gradient}. In particular, it was shown that the limit is a stationary biharmonic map, which implies that the singularity set of minimizing biharmonic maps is of co-dimension $5$. The interesting problem of whether this limit is minimizing remains open.

Thanks to the result of Scheven, we may study the tangent map at a singular point of a minimizing biharmonic map. The problem of uniqueness of such tangent maps is usually very difficult. Simon \cite{simon1983asymptotics} set up a general framework to attack such problem under a set of assumptions. The argument has been adapted to many different problems, for example, to minimal submanifolds \cite{simon1983asymptotics}, to Yang-Mills field \cite{yang2003uniqueness}, to Einstein metric \cite{colding2014uniqueness,cheeger1994cone}. To the best of our knowledge, all such generalizations are about the isolated singularity of solutions to some second order partial differential equation. It is the purpose of this paper to show that this argument also works in the case of fourth order problem. More precisely, we prove

\begin{thm}
	\label{thm:main}
	Suppose $N$ is an analytic submanifold of the Euclidean space $\Real^p$ and $u:B\to N$ is a minimizing biharmonic map (with finite energy), where $B\subset \Real^m (m\geq 5)$ is the unit ball. If $0$ is a singularity of $u$ and one of the tangent maps of $u$ at $0$ is of the form $\varphi(\frac{x}{\abs{x}})$ for some smooth $\varphi:S^{m-1}\to N$, then this tangent map is the {\it unique} tangent map at $0$.
\end{thm}

Suppose that $(r,\theta)$ is the polar coordinates in $B$ and that $t=-\log r$, then the theorem claims that $\lim_{t\to \infty} u(t)$ exists (and therefore is unique). As is well known, this is related to an estimate on the speed of convergence of $\partial_t u$ to zero when $t\to +\infty$. It is not hard to derive from the monotonicity formula (see (2.4) of \cite{scheven2008dimension} and \eqref{eqn:temp} in this paper) that
\begin{equation}\label{intro:l2}
	\int_1^{+\infty} \norm{\partial_t u}^2_{L^2(S^{m-1})} < +\infty.
\end{equation}
Here $S^{m-1}$ is the unit sphere in $\Real^m$.
If we can show
\begin{equation}\label{intro:l1}
	\int_1^{+\infty} \norm{\partial_t u}_{L^2(S^{m-1})} < +\infty,
\end{equation}
then we know at least $u(t)$ converges to a unique limit in the sense of $L^2$ norm.  However, in general, (\ref{intro:l1}) does not follow from (\ref{intro:l2}).

Simon \cite{simon1983asymptotics} observed that an infinite-dimensional version of Lojasiewicz inequality is helpful here. In the case of harmonic map, $u(t)$ is regarded as a family of maps from $S^{m-1}$ into $N$ evolving by some second order (abstract) ODE.
\begin{equation}\label{intro:ode}
	u''-(m-2)u'=\nabla \mathcal E_{S^{m-1}}(v) + R,
\end{equation}
where $u'=\partial_t u$, $\mathcal E_{S^{m-1}}$ is the harmonic map energy on $S^{m-1}$ and $R$ is some small perturbation term. A stationary point of this ODE (i.e. a solution independent of $t$) is the smooth map $\varphi$ in the assumptions of the theorem (in the harmonic map case). With the help of Lojasiewicz inequality, he studied the dynamics of this second order ODE in a small neighborhood of $\varphi$. More precisely, he proved (in Chapter 3 of his book \cite{simon1996theorems})
\begin{equation}\label{intro:odeinequality}
	\left( \int_{t+1}^{+\infty} \norm{\partial_t u}_{L^2(S^{m-1})}^2\right)^{2-\alpha} \leq  C \int_{t-1}^{t+1}\norm{\partial_t u}_{L^2(S^{m-1})}^2
\end{equation}
for any $t$ and some $\alpha\in (0,1)$. This amounts to (up to technical issue) an ordinary differential inequality of $h(t)=\int_t^{+\infty} \norm{\partial_t u}_{L^2(S^{m-1})}^2$,
\begin{equation*}
	h(t)^{2-\alpha}\leq C (-h'(t)).
\end{equation*}
From this inequality, it is easy to derive some decay estimate that implies (\ref{intro:l1}).

To generalize this argument to the biharmonic map case, we found that the Lojasiewicz inequality is not a problem because it is a general property of analytic functions, and the Lyapunov-Schmidt reduction works as long as the gradient of the functional is elliptic. The difficulty is to find the correct counter part of (\ref{intro:odeinequality}). We will eventually prove a discrete version of ordinary differential inequality with time delay (see \eqref{eqn:likesimon}).
Fortunately, we can still derive the decay estimate we need from it.

The paper is organized as follows. We recall some basic properties of biharmonic maps in Section \ref{sec:pre}. In particular, we prove an improved $\varepsilon$-regularity lemma of Schoen and Uhlenbeck type (see Proposition 4.5 of \cite{schoen1982regularity}). In Section \ref{sec:LS}, we prove the Lojasiewicz inequality (following \cite{simon1996theorems}). Section \ref{sec:growth} is the most important part of this paper, which contains the derivation of our analogue of (\ref{intro:odeinequality}). Finally, we give the proof of Theorem \ref{thm:main} in Section \ref{sec:proof} following the framework of Simon \cite{simon1996theorems}.

\section{Preliminaries on biharmonic maps}\label{sec:pre}
In this section, we collect a few results, mainly PDE estimates, that are needed for the proof of our main theorem.

We start by introducing the Euler-Lagrange equation for extrinsic biharmonic energy $E(u)$ (see Proposition 2.2 of \cite{wang2004stationary}),
\begin{equation}
	\triangle^2 u = \triangle(A(u)(\nabla u,\nabla u)) + 2 \nabla \cdot \langle \triangle u, \nabla(P(u))\rangle - \langle \triangle (P(u)),\triangle u\rangle.
	\label{eqn:EL}
\end{equation}
Here $A$ is the second fundamental form of $N$ in $\Real^p$ and $P(u)$ is the projection from $\Real^p$ to $T_uN$.
When $u$ is a smooth (extrinsic) biharmonic map, this is equivalent to the statement that $\triangle^2 u$ is perpendicular to $T_uN$ in $\Real^p$. Often in the following discussion, this simpler form is good enough.

\subsection{An improved $\varepsilon$-regularity}
The famous $\varepsilon$-regularity theorem for stationary harmonic maps requires that the (rescaled) energy is small on a ball. It has a biharmonic map analogue as follows:

\begin{lem}[\cite{wang2004stationary},\cite{struwe2008partial},\cite{scheven2008dimension}]\label{lem:epsilon} There exist $\varepsilon_1>0$ and constants $C(k)$ only depending on $N$ such that if $u$ is a stationary (extrinsic) biharmonic map on $B_r(x)\subset \Real^m (m\geq 5)$ satisfying
	\begin{equation}\label{eqn:smallness}
		r^{4-m}\int_{B_r(x)} (\abs{\nabla^2 u}^2 + r^{-2} \abs{\nabla u}^2) dx \leq \varepsilon_1,
	\end{equation}
	then
	\begin{equation*}
		\sup_{B_{r/2}(x)} r^k \abs{\nabla^k u} \leq C(k)\qquad \forall k\in \mathbb N.	
	\end{equation*}
\end{lem}
\begin{rem}
Here and throughout the paper, $B_r(x)$ means the ball of radius $r$ centered at $x$, which is usually omitted if $x=0$. Also the subscript $r$ is omitted if $r=1$.
\end{rem}

For minimizing harmonic maps, this result can be improved in the sense that a smallness condition on $\fint_{B_r(x)} \abs{u-u^*}^2 dx$ replaces \eqref{eqn:smallness}, where $u^*$ is the average of $u$ on $B_r(x)$ (see Proposition 4.5 of \cite{schoen1982regularity}). The improved version plays an important role in the analysis of minimal tangent maps and the uniqueness of tangent cones of harmonic maps. Therefore, we also need a biharmonic map version of it.

Since the extension lemmas in \cite{schoen1982regularity, luckhaus1988partial} are not available for biharmonic maps, the original proof in \cite{schoen1982regularity} does not work here. Fortunately, in Theorem 1.5 of \cite{scheven2008dimension}, Scheven proved that if $u_i$ is a sequence of minimizing biharmonic maps with bounded total energy, then there is a subsequence converging {\it strongly} to a stationary biharmonic map. More precisely, we have

\begin{lem}[Proposition 1.5 of \cite{scheven2008dimension}]
	\label{lem:scheven} Suppose that $u_i:B_2\to N$ is a sequence of minimizing biharmonic maps with bounded energy. Then there is a subsequence $u_{i_j}$ that converges {\emph strongly} to a stationary biharmonic map on $B_1$.
\end{lem}
\begin{proof}
	Since we have assumed that $N$ is compact, then $u_i$'s are uniformly bounded on $B_2$. The energy bound then implies that $\norm{u_i}_{W^{2,2}(B_{3/2})}$ is bounded, so that we can use Proposition 1.5 of \cite{scheven2008dimension} to get a subsequence converging to $u$ strongly in $W^{2,2}$. $u$ is stationary because the minimizers are stationary and the property of being stationary is preserved in strong limit.
\end{proof}

We then combine Lemma \ref{lem:scheven} and Lemma \ref{lem:epsilon} to get the biharmonic version of Proposition 4.5 of \cite{schoen1982regularity}.

\begin{lem}[biharmonic map version of Proposition 4.5 in \cite{schoen1982regularity}]\label{lem:regularity}
	For $\Lambda>0$ fixed, there is $\varepsilon_2=\varepsilon_2(N,\Lambda)>0$ such that the following holds. Suppose that $u:B_2\to N$ is a minimizing biharmonic map with $E(u)<\Lambda$ and $B_r(x)\subset B_1$. If
	\begin{equation*}
		r^{-m} \int_{B_r(x)} \abs{u-q}^2 dx \leq \varepsilon_2
	\end{equation*}
	for some $q\in N$, then
	\begin{equation*}
		\sup_{B_{r/4}(x)} r^k \abs{\nabla^k u}\leq C(k)\qquad \forall k\in \mathbb N,
	\end{equation*}
	for some $C(k)>0$.
\end{lem}

\begin{proof}
	By Lemma \ref{lem:epsilon}, it suffices to show
	\begin{equation*}
		(r/2)^{4-m} \int_{B_{r/2}(x)} \abs{\nabla^2 u}^2 + (r/2)^{-2}\abs{\nabla u}^2 dx \leq \varepsilon_1.
	\end{equation*}
	If otherwise, we have a sequence of minimizing biharmonic maps $u_i:B_2\to N$ with $E(u_i)<\Lambda$ such that for some $B_{r_i}(x_i)\subset B_1$, we have
	\begin{equation}\label{eqn:small}
	r_i ^{-m}\int_{B_{r_i}(x_i)} \abs{u_i-q_i}^2 dx \to 0
\end{equation}
and
\begin{equation}\label{eqn:large}
	(r_i/2)^{4-m} \int_{B_{r_i/2}(x)} \abs{\nabla^2 u_i}^2 + (r_i/2)^{-2} \abs{\nabla u_i}^2 dx \geq \varepsilon_1.
\end{equation}
Note that $r_i$ may converge to zero. Let $v_i(x)= u_i(r_i x + x_i)$. The monotonicity formula (see Lemma 5.3 of \cite{wang2004stationary}) tells us that
\begin{equation*}
	\int_{B_2(0)} \abs{\nabla^2 v_i}^2 + \abs{\nabla v_i}^2 dx \leq C(\Lambda).
\end{equation*}
By Lemma \ref{lem:scheven}, taking subsequence if necessary, $v_i$ converges to some stationary biharmonic map $v$ strongly in $W^{2,2}(B_1)$, which must be the trivial map due to (\ref{eqn:small}). Since the convergence is strong in $W^{2,2}$, we know that
\begin{equation*}
	\int_{B_1(0)} \abs{\nabla^2 v_i}^2+ \abs{\nabla v_i}^2 dx \to 0,
\end{equation*}
for $i$ sufficiently large. This is a contradiction with (\ref{eqn:large}) and therefore the lemma is proved.
\end{proof}

\subsection{Section of tangent cone}\label{sec:section}
Let $u$ be the minimizing biharmonic map in Theorem \ref{thm:main}. By the assumptions of the theorem, there is some sequence $r_i\to 0$ such that $u(r_i x)$ converges to a homogenous tangent map (which is biharmonic)
\begin{equation}\label{eqn:varphi}
	\tilde{\varphi}:=\varphi(\frac{x}{\abs{x}})
\end{equation}
and $\varphi$ is a smooth map from $S^{m-1}$ to $N$. It follows from Lemma \ref{lem:scheven} and Lemma \ref{lem:epsilon} that this convergence is in fact smooth convergence away from the origin.

Recall that in the harmonic map case, $\tilde{\varphi}$ is a harmonic map if and only if so is $\varphi$. Here for the biharmonic maps, the situation is somewhat different and it is the purpose of this subsection to characterize $\varphi$ that appears as the section of homogeneous biharmonic maps.

Let $(r,\theta)$ be the polar coordinates of $\Real^m$. A direct computation shows
\begin{eqnarray*}
	&& \triangle^2 \tilde{\varphi} \\
	&=& \left( \frac{\partial^2}{\partial r^2} + \frac{m-1}{r}\pfrac{}{r} + \frac{1}{r^2}\triangle_{S^{m-1}} \right)^2 \tilde{\varphi} \\
	&=& r^{-4}\left(  \triangle_{S^{m-1}}^2 \varphi -(2m-8) \triangle_{S^{m-1}}\varphi \right).
\end{eqnarray*}
If $\tilde{\varphi}$ is a biharmonic map, then $\triangle^2 \tilde{\varphi}\perp T_{\tilde{\varphi}}N$, which is equivalent to
\begin{equation}\label{eqn:section}
	\triangle_{S^{m-1}}^2 \varphi -(2m-8) \triangle_{S^{m-1}} \varphi \perp T_{\varphi}N.
\end{equation}
Instead of working out the explicit formula of (\ref{eqn:section}), it suffices for our purpose to note that it is the Euler-Lagrange equation of the energy functional
\begin{equation}\label{eqn:FF}
	F(\varphi):= \int_{S^{m-1}} \abs{\triangle_{S^{m-1}} \varphi}^2 + (2m-8) \abs{\nabla_{S^{m-1}} \varphi}^2 d\theta.
\end{equation}
Here we write $d\theta$ for the volume element on $S^{m-1}$ and $\varphi$ is a map from $S^{m-1}$ to $N$.

\subsection{$L^2$ closeness implies $C^5$ closeness}

Let $\varphi$ be the smooth section in Theorem \ref{thm:main}, which is a smooth critical map of $F$. We define
	\begin{equation*}
		\mathcal O_{L^2}(\sigma)=\left\{\psi:S^{m-1}\to N\, |\, \norm{\psi-\varphi}_{L^2(S^{m-1})}<\sigma  \right\}
	\end{equation*}
	and
	\begin{equation*}
		\mathcal O_{C^5}(\sigma)=\left\{\psi:S^{m-1}\to N\, |\, \norm{\psi-\varphi}_{C^5(S^{m-1})}<\sigma  \right\}.
	\end{equation*}

	Let $u$ be a smooth biharmonic map defined on $B\setminus \set{0}$ and $(t,\theta)$ be the cylinder coordinates. In this paper, we often regard $u(t)$ as a family of maps from $S^{m-1}$ to $N$. In the proof of our main theorem, these $u(t)$'s are often close to $\varphi$ in various sense. The next theorem roughly says that $L^2$-closeness (of $u(t)$) to $\varphi$ on some $t$-interval implies $C^5$-closeness in a smaller $t$-interval.

	\begin{lem}\label{lem:l2implyc3}
		For any $\sigma_1>0$, there is $\sigma_2>0$ (depending on $\sigma_1$, $\varphi$ and $N$) such that the following is true.
		Let $u(t,\theta)$ be as above. If
		\begin{equation*}
			u(s)\in \mathcal O_{L^2}(\sigma_2), \qquad \forall s\in (t_0-2,t_0+2),
		\end{equation*}
		then
		\begin{equation}\label{eqn:c3close}
			u(s)\in \mathcal O_{C^5}(\sigma_1),\qquad \forall s\in (t_0-1,t_0+1).
		\end{equation}
		Moreover, for some $C>0$ (depending on $\sigma_1$, $\varphi$ and $N$),
		\begin{equation}\label{eqn:dtbound}
			\sum_{k=1,2,3,4}\sum_{j=0}^{4-k} \abs{\partial_t^k \nabla_{S^{m-1}}^j u}(s,\theta)\leq C, \qquad \forall s\in (t_0-1,t_0+1).
		\end{equation}
	\end{lem}
	\begin{rem}\label{rem:obv}
		It is clear from the proof below that the lemma is still true for any $C^k$ neighborhood of $\varphi$ instead of $C^5$.
	\end{rem}

	\begin{proof}
		Although the lemma is stated in terms of $(t,\theta)$ coordinates, the proof is more clearly presented in the $(r,\theta)$ coordinates. By the scaling invariance of (\ref{eqn:EL}), we may assume that $t_0=2$ and study the equation (\ref{eqn:EL}) on $B_1\setminus B_{e^{-4}}$. By abuse of notation, we also write $\varphi$ for the function
		\begin{equation*}
			\varphi(r,\theta)= \varphi(\theta).
		\end{equation*}
		The assumption that $u(s)\in \mathcal O_{L^2}(\sigma_2)$ implies that there is a constant $C(\sigma_2)$ (satisfying $\lim_{\sigma_2\to 0} C(\sigma_2)=0$) such that
		\begin{equation}\label{eqn:L2bound}
			\int_{B_{1}\setminus B_{e^{-4}}} \abs{u-\varphi}^2 dx\leq C(\sigma_2).
		\end{equation}

		Since $\varphi$ is smooth, there is some constant $C_\varphi$ depending only on $\varphi$ such that
		\begin{equation}\label{eqn:uniformcont}
			\abs{\varphi (x)- \varphi(y)}\leq C_\varphi \abs{x-y}.
		\end{equation}
		for any $x,y\in B_{1}\setminus B_{e^{-4}}$.
		For some $y\in B_{e^{-1}}\setminus B_{e^{-3}}$, consider the ball $B_\sigma(y)\subset B_1\setminus B_{e^{-4}}$ for some $\sigma>0$ to be determined later. By \eqref{eqn:L2bound} and \eqref{eqn:uniformcont}, we have
		\begin{eqnarray*}
			&& \sigma^{-m} \int_{B_\sigma(y)} \abs{u(x)- \varphi(y)}^2 dx \\
			&\leq&  2 \sigma^{-m} \int_{B_\sigma(y)} \abs{u(x)- \varphi(x)}^2 dx +  2 \sigma^{-m} \int_{B_\sigma(y)} \abs{\varphi(x)-\varphi(y)}^2 dx  \\
			&\leq&  2\sigma^{-m} C(\sigma_2) + 2 \abs{B} C_\varphi^2 \sigma^2.
		\end{eqnarray*}
		Here $\abs{B}$ is the volume of the unit ball in $\Real^m$.

		Let $\varepsilon_2$ be the constant in Lemma \ref{lem:regularity}. We first take $\sigma$ small with $2 \abs{B} C_\varphi^2 \sigma^2< \varepsilon_2/2$ and then choose $\sigma_2$ sufficiently small so that $2\sigma^{-m}C(\sigma_2)<\varepsilon_2/2$. Hence, Lemma \ref{lem:regularity} gives
		\begin{equation}\label{eqn:C5bound}
			\norm{u}_{C^6(B_{e^{-1}}\setminus B_{e^{-3}})}\leq C,
		\end{equation}
		from which (\ref{eqn:dtbound}) follows.
		(\ref{eqn:c3close}) can be proved by interpolation between the $C^6$ bound \eqref{eqn:C5bound} and the $L^2$ bound \eqref{eqn:L2bound}, if we choose $\sigma_2$ smaller.
	\end{proof}

	\subsection{Estimates of $\partial_t u$}\label{sec:dtu}
Since being biharmonic is invariant under scaling and the group of scaling is generated by the vector field $r\partial_r=\partial_t$, if $u$ is a biharmonic map, then $\partial_t u$ satisfies the linearization equation of \eqref{eqn:EL}, which is a homogeneous linear elliptic system whose coefficients depend on $u$. Using this equation, we can prove

\begin{lem}
	\label{lem:dtestimate} If $u$ satisfies (\ref{eqn:c3close}) and $(\ref{eqn:dtbound})$ for $s\in (t_0-1,t_0+1)$, then we have
	\begin{equation}
		\sum_{k=1,2,3,4}\sum_{j=0}^{4-k} \abs{\partial_t^k \nabla_{S^{m-1}}^j u}^2(t_0,\theta)\leq \tilde{C} \int_{t_0-1}^{t_0+1} \int_{S^{m-1}} \abs{\partial_t u}^2 d\theta dt,
		\label{eqn:dtestimate}
	\end{equation}
	for some constant $\tilde{C}$ depending only on $\sigma_1$ (in \eqref{eqn:c3close}), $C$ (in \eqref{eqn:dtbound}) and the target manifold $N$.
\end{lem}
\begin{proof}
	The proof is interior estimate of elliptic system. By scaling invariance of (\ref{eqn:EL}), we may assume that $t_0=2$. Hence to show (\ref{eqn:dtestimate}), it suffices to prove
	\begin{equation*}
		\norm{\partial_r u}_{C^3(B_{e^{-3/2}}\setminus B_{e^{-5/2}})}\leq \tilde{C} \norm{\partial_r u}_{L^2(B_{e^{-1}}\setminus B_{e^{-3}})}.
	\end{equation*}
	The observation is that if we compute the homogeneous elliptic system of $r\partial_r u$ mentioned above, the H\"older norm of all coefficients are bounded due to (\ref{eqn:c3close}) and (\ref{eqn:dtbound}).
\end{proof}

\section{The Lojasiewicz-Simon inequality}\label{sec:LS}
The main purpose of this section is to prove the Lojasiewicz-Simon inequality for $F$ defined by \eqref{eqn:FF},
$$F(\psi)=\int_{S^{m-1}} |\triangle_{S^{m-1}}\psi|^2+(2m-8)|\nabla_{S^{m-1}}\psi|^2d\theta.$$
\begin{lem}
	\label{lem:loja}
	Let $\varphi$ be a smooth critical point of $F(\psi)$. Then there are $\varepsilon>0$, $\alpha\in (0,1]$ and $C>0$ depending on $\varphi$ such that for all $\psi:S^{m-1}\to N$ with
	\begin{equation*}
		\norm{\psi-\varphi}_{C^{5}(S^{m-1})}\leq \varepsilon,	
	\end{equation*}
	we have
	\begin{equation}
		\abs{F(\psi)-F(\varphi)}^{1-\alpha/2}\leq C \norm{\mathcal M_F(\psi)}_{L^2(S^{m-1})},
		\label{eqn:loja}
	\end{equation}
where $\mathcal M_F(\psi)$ is the Euler-Lagrange operator of $F$.
\end{lem}

\subsection{An equivalent form}
Since $\varphi$ is smooth, there is a natural correspondence between the maps that are close to $\varphi$ (in $C^5$ topology) and the small (in $C^5$ norm) sections of the pull-back bundle $V:=\varphi^* TN$. More precisely, we embed $N$ isometrically as a submanifold in $\Real^p$ and identify  a section $u$ of $\varphi^*TN$ with a map
\begin{equation*}
	u:S^{m-1}\to \Real^p, \qquad \mbox{satisfying} \quad u(\omega)\in T_{\varphi(\omega) }N\subset \Real^p.
\end{equation*}
Via the nearest point projection $\Pi$ defined in a tubular neighborhood of $N$, for any $\psi$ close to $\varphi$, we define $u$ by
\begin{equation}\label{eqn:upsi}
	\psi(\omega)= \Pi(\varphi(\omega)+ u(\omega)).
\end{equation}
This is well defined because for each $\omega$, $\Pi$ is a diffeomorphism between a neighborhood of $\varphi(\omega)$ in $N$ and a neighborhood of $0$ in $T_{\varphi(\omega)}N\subset \Real^p$.

Since $\varphi$ is a fixed smooth map, the $C^{k,\beta}$ norm of $u$ as a section of $V$ defined by the induced pull-back connection is equivalent to the $C^{k,\beta}$ norm of $u$ as a map from $S^{m-1}$ to $\Real^p$ (with restrictions to the image). The same applies to the Sobolev norms as well. While the intrinsic role of $u$ as a section is enough for the argument (the Lyapunov-Schmidt reduction), the extrinsic role of $u$ as a map is important in using the analyticity assumption. (See Appendix \ref{sec:analytic}.)

With the above identification in mind, define
\begin{equation*}
	\tilde{F}(u)= F(\psi)-F(\varphi)
\end{equation*}
 Then $\tilde{F}(0)=0$ and $u=0$ is a critical point of $\tilde{F}$. Let $\mathcal M_{\tilde{F}}(u)$ be the Euler-Lagrange operator of $\tilde{F}$ at $u$. Since the $L^2$ inner product of $V$ that we use to compute $\mathcal M_{\tilde{F}}$ is not identical to the $L^2$ inner product used for the computation of $\mathcal M_F$, $\mathcal M_F(\psi)$ is not trivially the same as $\mathcal M_{\tilde{F}}(u)$ with $u$ and $\psi$ related by \eqref{eqn:upsi}. However, we have
\begin{lem} \label{equivalenceofoldandnew}
	Let $\tilde{F}$ be defined as above, if $\psi$ is a map from $S^{m-1}$ to $N$ with $\|\psi-\varphi\|_{C^{4,\beta}}<\delta$ for sufficiently small $\delta>0$ and $u$ is defined by \eqref{eqn:upsi}, then
\begin{equation}\label{equ:equivalenceofoldandnew}
  (1-C\delta)|\mathcal{M}_{F}(\psi)|\leq |\mathcal{M}_{\tilde{F}}(u)|\leq |\mathcal{M}_{F}(\psi)|.
  \end{equation}
\end{lem}
The proof follows trivially from the equation (whose derivation is given in Appendix \ref{sec:analytic}, see \eqref{eqn:MM} there)
\begin{equation*}
	\mathcal M_{\tilde{F}}(u)=P_\varphi \mathcal M_F(\psi)	
\end{equation*}
and the fact that the tangent space $T_\psi N$ is close to $T_\varphi N$ since $\varphi$ is close to $\psi$.

Given Lemma \ref{equivalenceofoldandnew}, Lemma \ref{lem:loja} is reduced to
\begin{lem}
	\label{lem:loja2}
	There are $\varepsilon>0$, $\alpha\in (0,1]$ and $C>0$ depending on $\varphi$ such that for all $u\in C^5(V)$ with
	\begin{equation*}
		\norm{u}_{C^5(V)} \leq \varepsilon,	
	\end{equation*}
	we have
	\begin{equation}
		\abs{\tilde{F}(u)}^{1-\alpha/2}\leq C \norm{\mathcal M_{\tilde{F}}(u)}_{L^2(V)}.
		\label{eqn:loja2}
	\end{equation}
\end{lem}

\subsection{The Lyapunov-Schmidt reduction}\label{sec:reduction}
The proof of Lemma \ref{lem:loja2} is an application of the Lyapunov-Schmidt reduction argument. The local behavior of $\tilde{F}$ near $u=0$ is related to an analytic function defined on the finite dimensional kernel of an elliptic operator. More precisely, let $\mathcal L_{\tilde{F}}$ be the linearization of $\mathcal M_{\tilde{F}}$ at $u=0$, which is an elliptic operator from $C^{4,\beta}(V)$ to $C^{0,\beta}(V)$. By the theory of elliptic operators, the kernel of $\mathcal L_{\tilde{F}}$ is a finite dimensional space, denoted by $K$. Let $P_K$ be the orthogonal projection of $L^2(V)$ onto $K$.

Setting
\begin{equation*}
	\mathcal N(u)=P_K u + \mathcal M_{\tilde{F}}(u),
\end{equation*}
we find that $\mathcal N(0)=0$ and the linearization of $\mathcal N$ at $u=0$ is given by
\begin{equation*}
	P_K + \mathcal L_{\tilde{F}},
\end{equation*}
which is an isomorphism between $C^{4,\beta}(V)$ onto $C^{0,\beta}(V)$ because it is self-adjoint with trivial kernel. The inverse function theorem then gives an inverse $\Psi=\mathcal N^{-1}$ from a neighborhood of $0\in C^{0,\beta}(V)$ to $C^{4,\beta}(V)$.
\begin{rem}
	(1) For the ellipticity and self-adjointness of $\mathcal L_{\tilde{F}}$, we refer to Section \ref{subsec:properties}.	

	(2) The inverse function here actually appears as the real part of a complexified inverse function, which we need to justify the analyticity of $f$ in \eqref{eqn:f} below.
\end{rem}

Moreover, we have the following estimate for $\Psi$,
\begin{lem}\label{Lemma1}($L^{2}$ estimate)
 There is a neighborhood $W$ of $0$ in $C^{0,\beta}(V)$ and a constant $C$, depending only on $\tilde{F}$, such that
 $$ \|\Psi(f_{1})-\Psi(f_{2})\|_{W^{4,2}(V)}\leq C\|f_{1}-f_{2}\|_{L^{2}(V)},\qquad \text{for any} \quad f_1,\,f_2\in W.$$
\end{lem}
We refer to Appendix \ref{sec:Lemma1} for the proof.

With the help of $\Psi$, we define
\begin{equation}\label{eqn:f}
	f(\xi)= \tilde{F}(\Psi( \sum_{j=1}^l \xi^j \varphi_j))
\end{equation}
for $\abs{\xi}$ small, where $l=\dim K$ and $\set{\varphi_j}$ is a basis of $K$ with respect to the $L^2$ inner product.

The key to the proof of Lemma \ref{lem:loja2} and hence Lemma \ref{lem:loja} is the fact that $f$ is real analytic in a neighborhood of $0$. The proof relies on an analytic version of inverse function theorem for maps between complex Banach spaces and finally depends on the assumption about the analyticity of $N$ in Theorem \ref{thm:main}. It takes some efforts to be precise in tracing the use of this assumption and the details of this argument is given in Appendix \ref{sec:analytic}.

For now, we take the analyticity of $f$ near $0$ for granted. Therefore, it follows from the classical Lojasiewicz inequality that there are constants $\alpha\in (0,1]$, $C$ and $\sigma>0$ such that
\begin{equation}\label{eqn:lojaf}
	\abs{f(\xi)}^{(1-\frac{\alpha}{2})} \leq C \abs{\nabla f(\xi)},\qquad \text{for}\, \xi\in B_\sigma(0).
\end{equation}

For the proof of Lemma \ref{lem:loja2}, we need:
\begin{lem}
	\label{lem:facts} When $\norm{u}_{C^{4,\beta}(V)}$ is sufficiently small and hence $\xi^j=(u,\varphi_j)_{L^2}$ is small, we have
	\begin{equation}\label{eqn:ONE}
		\abs{\tilde{F}(u)-f(\xi)} \leq C \norm{\mathcal M_{\tilde{F}}(u)}_{L^2}^2;
	\end{equation}
	and
	\begin{equation}\label{eqn:TWO}
		\frac{1}{2} \abs{\nabla f}(\xi)\leq \norm{\mathcal M_{\tilde{F}}(\Psi(\sum_{j=1}^l \xi^j \varphi_j)}_{L^2} \leq 2 \abs{\nabla f}(\xi).
	\end{equation}
\end{lem}

Before the proof of Lemma \ref{lem:facts}, we show how Lemma \ref{lem:loja2} follows from it and \eqref{eqn:lojaf}.

In fact, by plugging \eqref{eqn:ONE} and \eqref{eqn:TWO} directly into \eqref{eqn:lojaf}, we get
\begin{equation}\label{eqn:weget}
	\begin{split}
	\abs{\tilde{F}(u)}^{1-\alpha/2} \leq& C (\norm{\mathcal M_{\tilde{F}}(\Psi(\sum \xi^j \varphi_j))}_{L^2} + \norm{\mathcal M_{\tilde{F}}(u)}_{L^2}^{2-\alpha}) \\
	\leq& C (\norm{\mathcal M_{\tilde{F}}(\Psi(\sum \xi^j \varphi_j))}_{L^2} + \norm{\mathcal M_{\tilde{F}}(u)}_{L^2}).
	\end{split}
\end{equation}
Here in the last line above, we use the facts that $2-\alpha\geq 1$ and that $\norm{\mathcal M_{\tilde{F}}(u)}$ is bounded for $u$ in the lemma. The first term in the right hand side of \eqref{eqn:weget} is dominated by the second, because
\begin{equation}\label{eqn:used}
\begin{split}
	& \norm{\mathcal M_{\tilde{F}}(\Psi(\sum \xi^j \varphi_j)) - \mathcal M_{\tilde{F}}(u)}_{L^2} \\
	\leq& C \norm{\Psi(\sum\xi^j \varphi_j)-u}_{W^{4,2}} \\
	\leq& C \norm{\sum\xi^j \varphi_j- \Psi^{-1} u}_{L^2} \\
	\leq& C \norm{\mathcal M_{\tilde{F}}(u)}_{L^2}.
\end{split}
\end{equation}
Here for the second line above, we noticed that $\mathcal M_{\tilde{F}}$ is a (nonlinear) fourth order differential operator (see \eqref{eqn:MM})and both the $C^{4,\beta}$ norms of $\Psi(\sum \xi^j \varphi_j)$ and $u$ are bounded; for the third line above, we used Lemma \ref{Lemma1}; for the last line, we used the definition of $\mathcal N =\Psi^{-1}$ and $P_K u= \sum \xi^j \varphi_j$.
Now, Lemma \ref{lem:loja2} is a consequence of \eqref{eqn:weget} and \eqref{eqn:used}.

The rest of this section is devoted to the proof of Lemma \ref{lem:facts}.
\subsection{The proof of Lemma \ref{lem:facts}}

By the definition of $f$ (see \eqref{eqn:f}) and $\xi$ (in the assumption of Lemma \ref{lem:facts}), $f(\xi)=\tilde{F}(\Psi(P_K u))$. Hence, to prove \eqref{eqn:ONE}, we compute
       \begin{eqnarray*}
       & &|\tilde{F}(u)-\tilde{F}(\Psi(P_{K}u))|\\
       &=& \left | \int_{0}^{1} \frac{d}{ds}\tilde{F}( u+s(\Psi(P_{K}u)-u))\,ds\right|\\
       &=& \left|\int_{0}^{1}\left(\mathcal{M}_{\tilde{F}}( u+s(\Psi(P_{K}u)-u)) ,\Psi(P_{K}u)-u \right)_{L^{2}} ds \right|.
       \end{eqnarray*}
       Again, by the facts that $\mathcal M_{\tilde{F}}$ is a fourth order operator and that $C^{4,\beta}$ norms of $u$ and $u+s(\Psi(P_K u)-u)$ are bounded for any $s\in [0,1]$, we have
       \begin{equation*}
       \| \mathcal{M}_{\tilde{F}}( u+s(\Psi(P_{K}u)-u))-\mathcal{M}_{\tilde{F}}( u)\|_{L^{2}}\leq C \| \Psi(P_{K}u)-u\|_{W^{4,2}},
       \end{equation*}
which implies that
\begin{eqnarray*}
	&& \abs{\tilde{F}(u)-\tilde{F}(\Psi(P_K u))} \\
	&\leq& C\norm{\Psi(P_K u)-u}_{L^2} \left( \norm{\mathcal M_{\tilde{F}}(u)}_{L^2} + \norm{\Psi(P_K u)-u}_{W^{4,2}} \right) \\
	&\leq& C \norm{\mathcal M_{\tilde{F}}(u)}_{L^2}^2.
\end{eqnarray*}
Here in the last line above, we used
\begin{equation*}
	\norm{\Psi(P_K u)-u}_{W^{4,2}}\leq C \norm{\mathcal M_{\tilde{F}}(u)}_{L^2},
\end{equation*}
which appeared as a part of \eqref{eqn:used} and was proved there. This concludes the proof of \eqref{eqn:ONE}.

For the proof of \eqref{eqn:TWO}, we compute using \eqref{eqn:f} and the chain rule to get
\begin{equation}\label{eqn:nablaf}
	(\eta,\nabla f(\xi))_{\Real^l} = \left( \mathcal M_{\tilde{F}}(\Psi(\sum \xi^j \varphi_j)), d\Psi|_{\sum \xi^j \varphi_j} (\sum \eta^j \varphi_j) \right)_{L^2},
\end{equation}
for some $\eta\in \Real^l$ with $\abs{\eta}=1$.

Notice that $d\Psi|_{\sum \xi^j \varphi_j}$ depends smoothly on $\xi$ in a compact neighborhood of $\xi=0$, hence there is $C>0$ such that
\begin{equation}\label{eqn:psi1}
	\norm{d\Psi|_{\sum \xi^j \varphi_j}-d\Psi|_0} \leq C \abs{\xi}, \quad \text{for small} \, \abs{\xi}.
\end{equation}
\begin{rem}
	(1) For the smooth dependence in $\xi$, we shall prove in the Appendix \ref{sec:analytic} that $\Psi$ has a complexification that is analytic (hence smooth by Theorem \ref{thm:smooth}).

	(2) The norm in \eqref{eqn:psi1} should be the norm of bounded linear operator from $C^\beta(V)$ to $C^{4,\beta}(V)$, according to our discussion in the appendix. What we need here is the inequality
	\begin{equation*}
		\norm{(d\Psi|_{\sum \xi^j \varphi_j}-d\Psi|_0)(\sum \eta^j \varphi_j)}_{L^2} \leq C \abs{\xi} \norm{\sum\eta^j \varphi_j}_{L^2}.
	\end{equation*}
	This is true because $\sum \eta^j\varphi_j$ lies in $K$ and when restricted to the finite dimensional space $K$, $L^2$ norm is equivalent to $C^\beta$ norm.
\end{rem}
On the other hand,
\begin{equation}\label{eqn:psi2}
	d\Psi|_0 (\sum \eta^j \varphi_j)= \sum \eta^j \varphi_j,\qquad \text{for any} \quad \eta\in \Real^l,
\end{equation}
because $d\Psi|_0= (d\mathcal N|_0)^{-1}= (P_K + \mathcal L_{\tilde{F}})^{-1}$, and $\sum \eta^j \varphi_j$ is in $K$, the kernel of $\mathcal L_{\tilde{F}}$.

By \eqref{eqn:psi1} and \eqref{eqn:psi2}, \eqref{eqn:nablaf} implies that
\begin{equation}\label{eqn:compare}
	\abs{ (\eta,\nabla f(\xi))_{\Real^l} - \left( \mathcal M_{\tilde{F}}(\Psi(\sum \xi^j \varphi_j)), \sum \eta^j \varphi_j \right)_{L^2}} \leq C \abs{\xi} \norm{\mathcal M_{\tilde{F}}(\Psi(\sum \xi^j \varphi_j))}_{L^2}.
\end{equation}

Now, in \eqref{eqn:compare}, if we choose $\eta$ parallel to $\nabla f(\xi)$ in $\Real^l$, we obtain
\begin{equation*}
	\abs{\nabla f} \leq (1+C\abs{\xi}) \norm{\mathcal M_{\tilde{F}}(\Psi(\sum \xi^j \varphi_j))}_{L^2};
\end{equation*}
if we choose $\eta$ so that $\sum \eta^j \varphi_j$ is parallel to $\mathcal M_{\tilde{F}}(\Psi(\sum \xi^j \varphi_j))$ (which is in $K$), then we get
\begin{equation*}
	(1-C\abs{\xi}) \norm{\mathcal M_{\tilde{F}}(\Psi(\sum \xi^j \varphi_j))}_{L^2}\leq \abs{\nabla f}.
\end{equation*}
This finishes the proof of \eqref{eqn:TWO} and hence Lemma \ref{lem:facts} if $\xi$ is small.
\section{Dynamics near a critical point of $F$}\label{sec:growth}
Let $u$ be the minimizing biharmonic map given in Theorem \ref{thm:main}. Recall that $(r,\theta)$ is the polar coordinates and that $t=-\log r$. By the assumptions of the theorem, there exists $t_i\to \infty$ such that $u(t_i,\theta)$ as maps on $S^{m-1}$ converges smoothly to a critical point $\varphi$ of $F$. (See the discussion in Section \ref{sec:section}.)

Therefore, for $i$ sufficiently large, $u(t_i,\theta)$ is very close in $C^5$ topology to the critical point $\varphi$ of $F$. Since $u$ is a biharmonic map, the biharmonic map equation determines how $u(t)$ should change as a map on $S^{m-1}$. In this section, we study this dynamics of $u(t)$ in a very small neighborhood of $\varphi$. More precisely, we are interested in the speed of decay of
\begin{equation*}
	\int_t^\infty \int_{S^{m-1}} \abs{\partial_t u}^2 d\theta dt
\end{equation*}
as explained in the introduction.
In fact, we shall control the decay of a larger quantity, namely,
\begin{equation}\label{eqn:defG}
	G(t)= \int_t^\infty \int_{S^{m-1}} (2m-8) \abs{\partial_t^2 u}^2 + (2m-8) \abs{\partial_t \nabla_{S^{m-1}} u}^2 + (2m-8)(m-2) \abs{\partial_t u}^2 d\theta dt.
\end{equation}

\begin{lem}
	\label{lem:dynamics} Suppose $\varphi$ is a smooth critical point of $F$. There is some constant $\sigma>0$ (depending on $\varphi$) such that if $u(t,\theta)$ (cylinder coordinates) is a smooth biharmonic map satisfying
	\begin{equation*}
		\norm{u(s)-\varphi}_{C^5(S^{m-1})}\leq \sigma, \qquad \text{for} \quad s\in [t-3,t+3],
	\end{equation*}
	then there exist $C'>0$ and $\theta\in (0,1)$ such that
	\begin{equation}\label{eqn:discreteode}
		G(s-1)^\theta -G(s+1)^{\theta} \geq C'\left( G(s-1)-G(s+1) \right)^{1/2}.
	\end{equation}
\end{lem}

Before we start the proof, we rewrite $\triangle^2 u$ in $(t,\theta)$ coordinates and split it into two parts.
Since
\begin{equation*}
	\triangle u = e^{2t}\left( \partial_t^2 - (m-2)\partial_t + \triangle_{S^{m-1}}\right) u,
\end{equation*}
we have
\begin{eqnarray}
	\nonumber
	&& \triangle^2 u\\
	\nonumber
	&=& e^{4t} \left( \partial_t^2 + \triangle_{S^{m-1}} - (m-6) \partial_t + (8-2m) \right) (\partial_t^2 -(m-2)\partial_t + \triangle_{S^{m-1}})u \\
	\nonumber
	&:=& e^{4t} (I_{a}+ I_{b}),
	\label{eqn:one}
\end{eqnarray}
where
\begin{eqnarray*}
	I_a&=&  \partial_t^4 u +  2\partial_t^2 \triangle_{S^{m-1}} u - (2m-8) \partial_t^3 u - (2m-8) \partial_t \triangle_{S^{m-1}} u\\
	&&  + (m^2-10m+20)\partial_t^2 u + (2m-8)(m-2)\partial_t u
\end{eqnarray*}
and
\begin{equation*}
	I_b=\triangle_{S^{m-1}}^2 u + (8-2m) \triangle_{S^{m-1}}u.
\end{equation*}
The idea behind this splitting is that we put everything involving $\partial_t$ in $I_a$ and the rest in $I_b$. An easy observation is that $I_b$ is almost (up to a projection) the gradient of $F$ discussed in Section \ref{sec:section}, namely,
\begin{equation}\label{eqn:dtF}
	\begin{split}
	\partial_t F(u(t)) &= 2 \int_{S^{m-1}} \triangle_{S^{m-1}} u \triangle_{S^{m-1}} \partial_t u + (2m-8) \nabla_{S^{m-1}} u \cdot \nabla_{S^{m-1}} \partial_t u d\theta\\
	&= 2 \int_{S^{m-1}} I_b\cdot \partial_t u d\theta.
	\end{split}
\end{equation}

The way we use the biharmonic map equation has nothing to do with the right hand side of (\ref{eqn:EL}). We multiply the equation by $\partial_t u$ and integrate over $S^{m-1}$, to obtain
\begin{equation}\label{eqn:useequation}
	0=\int_{S^{m-1}} \triangle^2 u \cdot \partial_t u d\theta = \int_{S^{m-1}} (I_a +I_b) \cdot \partial_t u d\theta.
\end{equation}

While $\int_{S^{m-1}} I_b \cdot \partial_t u d\theta$ is known in \eqref{eqn:dtF}, the structure of $\int_{S^{m-1}} I_a\cdot \partial_t u d\theta$ is still complicated. There is some positivity hidden in it. To reveal it, we use the elementary equalities
\begin{equation*}
	\partial_t^4 u \cdot \partial_t u = \partial_t \left( \partial_t^3 u \partial_t u -\frac{1}{2} \abs{\partial_t^2 u}^2 \right)
\end{equation*}
and
\begin{equation*}
	\partial_t^3 u \cdot \partial_t u = \partial_t \left( \partial_t^2 u\cdot \partial_t u \right) - \abs{\partial_t^2 u}^2,
\end{equation*}
to get
\begin{eqnarray*}
	&& \int_{S^{m-1}} I_a\cdot \partial_t u d\theta \\
	&=& \partial_t \left( \int_{S^{m-1}} \partial_t^3 u\partial_t u -\frac{1}{2}\abs{\partial_t^2 u}^2 - \abs{\partial_t \nabla_{S^{m-1}} u}^2 - (2m-8) \partial_t^2 u \partial_t u + \frac{m^2-10m+20}{2} \abs{\partial_t u}^2  d\theta\right)\\
	&& + \left( \int_{S^{m-1}} (2m-8) \abs{\partial_t^2 u}^2 + (2m-8) \abs{\partial_t \nabla_{S^{m-1}} u}^2 + (2m-8)(m-2) \abs{\partial_t u}^2d\theta \right) \\
	&:=& \partial_t \left( \int_{S^{m-1}} II_a d\theta  \right) + \int_{S^{m-1}} II_b d\theta.
\end{eqnarray*}
Notice that $II_b$ is nonnegative and this is how we obtain the definition of $G(t)$ in \eqref{eqn:defG}, i.e. $G(t)=\int_t^\infty \int_{S^{m-1}} II_b d\theta dt$.

By \eqref{eqn:useequation} and \eqref{eqn:dtF}, we have
\begin{equation}\label{eqn:wehave}
	\frac{1}{2}\partial_t F(u(t)) = -\int_{S^{m-1}}I_a\cdot \partial_t u d\theta.
\end{equation}
Let $t_i$ be the sequence mentioned in the beginning of this section such that $u(t_i)$ converges smoothly to the smooth section map $\varphi$. Moreover, $u(t+t_i)$ regarded as a map defined on $[-1,1]\times S^{m-1}$ converges smoothly to $\tilde{\varphi}(t,\theta)=\varphi(\theta)$. This implies that
\begin{equation*}
	\lim_{i\to \infty} \int_{S^{m-1}} II_a(t_i) d\theta =0,
\end{equation*}
so that if we integrate \eqref{eqn:wehave} from $s$ to $t_i$ and take the limit $i\to \infty$, we obtain
\begin{equation}\label{eqn:temp}
	\frac{1}{2}\left( F(\varphi)-F(u(s)) \right) = \int_{S^{m-1}} II_a (s) d\theta - \int_s^{+\infty} \int_{S^{m-1}} II_b d\theta.
\end{equation}
As a by-product of the above computation, $G(t)$ is a finite number, which is the biharmonic counterpart of \eqref{intro:l2}.

We may choose $\sigma$ small so that for $u$ in the lemma and $s\in [t-3,t+3]$, $\norm{u(s)-\varphi}_{C^5(S^{m-1})}$ is small and hence $u(s)$ satisfies the assumption of Lemma \ref{lem:loja}. The Lojasiewicz-Simon inequality (in Lemma \ref{lem:loja}) and \eqref{eqn:temp} imply that
\begin{equation}
	\label{eqn:loj1}
	- \int_{S^{m-1}}II_a(s)d\theta + \int_s^{ +\infty} \int_{S^{m-1}} II_b d\theta dt \leq C \norm{\mathcal M_F(u(s))}_{L^2(S^{m-1})}^{\frac{2}{2-\alpha}}
\end{equation}
for some $\alpha\in (0,1]$.

Next, we show that the right hand side and the first term in the left hand side of \eqref{eqn:loj1} are controlled by $\int_{s-1}^{s+1} \abs{II_b}^2 d\theta$. To see this, recall that by \eqref{eqn:dtF}, $\mathcal M_F(u(s))$ is the projection of $I_b(s)$ onto the tangent bundle of $TN$ at $u(s)$. If we denote this projection from $\Real^p$ onto $T_u N$ by $\Pi$,
\begin{equation}\label{eqn:proj1}
	\mathcal M_F(u(s))= 2 \Pi (I_b(s)).
\end{equation}
On the other hand, since $u$ is extrinsic biharmonic map, the Euler-Lagrange equation reads
\begin{equation}\label{eqn:proj2}
	\Pi (\triangle^2 u) =\Pi(I_a+ I_b)=0.
\end{equation}
Combining \eqref{eqn:proj1} and \eqref{eqn:proj2}, we get
\begin{equation}\label{eqn:hao1}
	\norm{\mathcal M_F(u(s))}_{L^2(S^{m-1})}\leq 2 \norm{I_a(s)}_{L^2(S^{m-1})}.
\end{equation}

Notice that the integrands of both $I_a(s)$ and $II_a(s)$ involve $\partial_t u$ and its derivatives, which are estimated in Section \ref{sec:dtu}. More precisely, by taking $\sigma$ small, we may apply Lemma \ref{lem:l2implyc3} first to get \eqref{eqn:dtbound} and then Lemma \ref{lem:dtestimate} to see
\begin{equation}\label{eqn:hao2}
		\norm{I_a(s)}_{L^2(S^{m-1})}^2 + \int_{S^{m-1}} \abs{II_a}(s) d\theta \leq C \int_{s-1}^{s+1} \int_{S^{m-1}}\abs{II_b}^2 d\theta dt.
\end{equation}

By the definition of $G(t)$ in \eqref{eqn:defG}, \eqref{eqn:loj1}, \eqref{eqn:hao1} and \eqref{eqn:hao2} imply
\begin{equation*}
	-C (G(s-1)-G(s+1))+ G(s) \leq C \left( G(s-1)-G(s+1) \right)^{\frac{1}{2-\alpha}}.
\end{equation*}
Since $G(s-1)-G(s+1)$ is bounded and $\frac{1}{2-\alpha}\leq 1$, the first term can be absorbed into the left hand side. In fact, in the proof that follows, we shall require $G(s)$ to be very small (see the definition of $\eta$ in the next section).
By the monotonicity of $G$, the above inequality is further simplified to
\begin{equation}\label{eqn:likesimon}
	G(s+1) \leq C \left( G(s-1)-G(s+1) \right)^{\frac{1}{2-\alpha}}.
\end{equation}

Here is a lemma similar to (9) in Section 3.15 of  \cite{simon1996theorems}.
\begin{lem}\label{lem:discrete}
	Suppose that $\theta\in (0,1/2]$. If for some positive $C$ and any $a,b\in (0,1)$ satisfying $b<a$,
	\begin{equation}\label{eqn:abC}
		b\leq C (a-b)^{1/(2-2\theta)},
	\end{equation}
	then there is another $C'$ depending only on $C$ and $\theta$ such that
	\begin{equation*}
		a^{\theta}-b^{\theta}\geq C'(a-b)^{1/2}.
	\end{equation*}
\end{lem}

\begin{proof} The proof is an elementary discussion.

Case 1: $b<\frac{1}{2}a$. Noticing that $\theta\leq 1/2$ and $a<1$, we have
\begin{equation*}
	a^\theta-b^\theta\geq (1-\frac{1}{2^\theta}) a^\theta\geq (1-\frac{1}{2^\theta}) a^{1/2} \geq (1-\frac{1}{2^\theta}) (a-b)^{1/2}.
\end{equation*}

Case 2: $b\geq \frac{1}{2}a$. \eqref{eqn:abC} gives
\begin{equation*}
	\frac{a}{2C} \leq (a-b)^{\frac{1}{2(1-\theta)}},
\end{equation*}
which is
\begin{equation}\label{eqn:btheta}
	a^{1-\theta}\leq (2C)^{1-\theta} (a-b)^{1/2}.
\end{equation}
Therefore,
\begin{equation*}
	a^\theta-b^\theta\geq \theta a^{\theta-1}(a-b) \geq \frac{\theta}{(2C)^{1-\theta}}(a-b)^{1/2}.
\end{equation*}
Here in the above line we have used the mean value theorem for the first inequality and \eqref{eqn:btheta} for the second.

In either case, the lemma is proved by taking $C'$ to be $\min\set{1-\frac{1}{2^\theta},\frac{\theta}{(2C)^{1-\theta}}}$.
\end{proof}

\section{A stability argument and the proof of Theorem \ref{thm:main}}\label{sec:proof}
In this section, we prove Theorem \ref{thm:main} by using a routine stability argument. We shall define two neighborhoods of $\varphi$: a larger one (see $\mathcal O_{C^5}(\sigma_1)$ below) in which the results in Section \ref{sec:LS} and Section \ref{sec:growth} hold and a smaller one (see $\mathcal O_{L^2}(\eta)$ below) such that if $u(t_i)$ lies in the smaller neighborhood for sufficiently large $i$, then $u(t)$ will stay in the larger neighborhood forever and converge to the unique limit claimed in Theorem \ref{thm:main}.

We choose $\sigma_1$ so that it is smaller than both the $\varepsilon$ in Lemma \ref{lem:loja} and the $\sigma$ in Lemma \ref{lem:dynamics}. For some $\eta>0$ small (to be determined later), by the definition of $\varphi$ as the section of a tangent map, we can choose (and fix) $t_i$ large such that

1) for all $t\in (t_i-3,t_i+3)$, $u(t)\in \mathcal O_{C^5}(\sigma_1)$;

2) $u(t_i)\in \mathcal O_{L^2}(\eta)$;

3) $G(t_i) \leq \eta^2$, because $G(t)$ is finite and decreases down to zero .

Set
\begin{equation*}
	T=\sup_{t}\left\{ t\,|\, \mbox{for any} \, s\in [t_i,t), \, u(s)\in \mathcal O_{C_5}(\sigma_1) \right\}.
\end{equation*}
By 1) above, we know $T\geq t_i+3$. Now we claim that $T$ is infinity. If otherwise, we want to find a contradiction by showing $u(T)\in \mathcal O_{C^5}(\sigma_1/2)$. Thanks to Lemma \ref{lem:l2implyc3}, there is $\sigma_2>0$ depending on $\sigma_1/2$ such that it suffices to show for any $s\in (t_i,T+2)$, we have $u(s)\in \mathcal O_{L^2}(\sigma_2)$. Let $k$ be the largest integer with $t_i+2k\leq s$. Hence,
\begin{eqnarray*}
	\int_{t_i}^s \norm{\partial_t u}_{L^2(S^{m-1})} &\leq& \sum_{j=1}^k \int_{t_i+2(j-1)}^{t_i+2j} \norm{\partial_t u}_{L^2(S^{m-1})} + \int_{t_i+2k}^s \norm{\partial_t u}_{L^2(S^{m-1})} \\
	&\leq&  C \sum_{j=1}^k \left( \int_{t_i+2(j-1)}^{t_i+2j} \norm{\partial_t u}_{L^2(S^{m-1})}^2 \right)^{1/2} +  C \left( \int_{t_i+2k}^s \norm{\partial_t u}_{L^2(S^{m-1})}^2 \right)^{1/2} \\
	&\leq&  C \sum_{j=1}^k \left( \int_{t_i+2(j-1)}^{t_i+2j} \norm{\partial_t u}_{L^2(S^{m-1})}^2 \right)^{1/2} +  C \eta.
\end{eqnarray*}
Here in the second line above, we used H\"older inequality and in the last line, we used 3).

By the definition of $G$, we have
\begin{equation*}
	\int_{t_i+2(j-1)}^{t_i+2j} \norm{\partial_t u}_{L^2(S^{m-1})}^2\leq G(t_i+2(j-1))- G(t_i+2j).
\end{equation*}
We can apply Lemma \ref{lem:discrete} with $a= G(t_i+2j)$ and $b=G(t_i+2(j-1))$ to get
\begin{eqnarray*}
	\int_{t_i}^s \norm{\partial_t u}_{L^2(S^{m-1})} &\leq& C \sum_{j=1}^k\left( G(t_i+2(j-1))^\theta - G(t_i+2j)^{\theta} \right) + C\eta \\
	&\leq& C \cdot G(t_i)^\theta + C\eta \\
	&\leq& C \eta^{2\theta}+ C\eta.
\end{eqnarray*}
If we choose $\eta$ small, we can have for any $s\in (t_i,T+2)$,
\begin{equation*}
	\norm{u(s)-\varphi}_{L^2(S^{m-1})}\leq \norm{u(t_i)-\varphi}_{L^2(S^{m-1})}+ \int_{t_i}^s \norm{\partial_t u}_{L^2(S^{m-1})} \leq \sigma_2/2.
\end{equation*}
Lemma \ref{lem:l2implyc3} gives the contradiction and proves that $T=\infty$.

We can repeat the above computation with $k=\infty$ to get
\begin{equation*}
	\int_{t_i}^{+\infty} \norm{\partial_t u}_{L^2(S^{m-1})} \leq C \eta^{2\theta}+ C\eta< \infty,
\end{equation*}
which shows that
\begin{equation*}
	\lim_{t\to \infty} \norm{u(t)-\varphi}_{L^2(S^{m-1})}=0.
\end{equation*}
As in Remark  \ref{rem:obv}, we have $u$ bounded in any $C^{k+1}(S^{m-1})$ norm. By interpolation, we know
\begin{equation*}
	\lim_{t\to \infty} \norm{u(t)-\varphi}_{C^k(S^{m-1})}=0.
\end{equation*}
\appendix

\section{The assumption of analyticity}\label{sec:analytic}
The purpose of this section is to justify (see Lemma \ref{lem:justify}) the use of the classical Lojasiewicz inequality to the function $f$ (see \eqref{eqn:f}) that arises in the Lyapunov-Schmidt reduction in Section \ref{sec:reduction}. Indeed, we shall show how the analyticity assumption of $N$ in Theorem \ref{thm:main} carries on step by step to that of $f$. These arguments, independent from the rest of the proof, are technical and hence presented in the appendix.

\subsection{Analytic function between Banach spaces}\label{subsec:background}
For completeness, we collect a few basic definitions and properties of analytic functions between abstract (complex) Banach spaces. We refer to Taylor's paper \cite{taylor1937} for proofs and more detailed discussions.

Let $E,E'$ and $E''$ be complex Banach spaces.

\begin{defn}\label{defn:analytic}
	(1) Let $f(x)$ be a function on $E$ to $E'$, defined in the neighborhood of $x_0\in E$. If for each $y\in E$, the limit
	\begin{equation*}
		\lim_{\tau\to 0} \frac{f(x_0+\tau y)-f(x_0)}{\tau}
	\end{equation*}
	exists (for $\tau\in \mathbb C$), then it is called the {\bf Gateaux} differential, denoted by $\delta f(x_0;y)$.
	
	(2) A function $f(x)$ on a domain $D$ of $E$ to $E'$ is said to be {\bf analytic} in $D$ if it is continuous and has a Gateaux differential at each point of $D$. A function is said to be analytic at a point $x_0$, if it is analytic in some neighborhood of $x_0$.
\end{defn}

Recall that the {\bf Fr\'echet} differential is defined to be the bounded linear map $Df(x_0)$ from $E$ to $E'$ such that
\begin{equation*}
	f(x_0+h)=f(x_0)+ Df(x_0) h + o(\norm{h}_{E}).
\end{equation*}
While the existence of Fr\'echet differential is obviously stronger than the Gateaux differential, Taylor proved
\begin{thm}\label{thm:smooth}
	[Theorem 3 and Theorem 12 in \cite{taylor1937}] If $f$ is analytic at $x_0$, then it admits Fr\'echet differentials of all orders in the neighborhood of that point. Moreover, the Fr\'echet differential and the Gateaux differential are equal.
\end{thm}
With the equivalence in mind, we recall a version of Inverse Function Theorem, which follows from (10.2.5) in the book of Dieudonn\'e \cite{dieudonne} (see also Section 2.7 of \cite{nirenbergnonliearfunctionalanalysis}).

\begin{thm}
	\label{thm:IFT} Let $E$ and $E'$ be two complex Banach spaces, $f$ an analytic function from a neighborhood $V$ of $x_0\in E$ to $E'$. If $Df(x_0)$ is a linear homeomorphism of $E$ onto $E'$, there exists an open neighborhood $U\subset V$ of $x_0$ such that the restriction of $f$ to $U$ is a homeomorphism of $U$ onto an open neighborhood of $y_0=f(x_0)$. Moreover, the inverse is analytic.
\end{thm}

\subsection{Complexification and analyticity}\label{subsec:complexification}
In section \ref{sec:LS}, we have defined the functional $\tilde{F}: C^{4,\beta}(V)\to \Real$ where $V$ is the pullback bundle $\varphi^* TN$ and a map $\mathcal N$ from $C^{4,\beta}(V)$ to $C^{0,\beta}(V)$. Instead of claiming the analyticity of $\tilde{F}$ and $\mathcal N$ directly, we consider its complexification.

$C^{4,\beta}(V)\otimes \mathbb C$ is understood to be the set of $u+iv$, where $u,v\in C^{4,\beta}(V)$, with a naturally defined norm. The same applies to $C^{0,\beta}(V)\otimes \mathbb C$. Obviously, they are complex Banach spaces.

A complexification of a map $f$ from a Banach space $E_1$ to another Banach space $E_2$ is some map $\tilde{f}$ from $E_1\otimes \mathbb C$ to $E_2\otimes \mathbb C$ such that $f$ is the real part of $\tilde{f}$ when restricted to (some open set of) $E_1$. Such complexifications are usually not unique. We are interested in analytic ones, that we define below (making using of special properties of $f$).

The complexification of $\tilde{F}$ and $\mathcal N$ relies on some particular form of the maps themselves. More precisely, we need the definition of $\tilde{F}(u)$ and $\mathcal N(u)$ to be given by a converging series. For this purpose, we start with an extrinsic point of view of $V$.

Since $N$ is embedded in $\Real^p$, we regard $T_y N$ as a subspace (not the affine space passing $y$) of $\Real^p$. Hence, the pullback bundle $V$ is the disjoint union of $V_\omega:=T_{\varphi(\omega)} N$ and a section $u$ of $V$ is a map from $S^{m-1}$ to $\Real^p$ satisfying
\begin{equation*}
	u(\omega)\in T_{\varphi(\omega)} N \subset \Real^p.
\end{equation*}
For a fixed smooth $\varphi$, the $C^{k,\beta}$ norm of $u$ as a map into $\Real^p$ agrees with the $C^{k,\beta}$ norm defined intrinsically using the pullback connection of $\varphi^*TN$. The same holds for various Sobolev norms.

For the complexification of $\tilde{F}$, we regard it as the composition of
\begin{equation*}
	C^{2,\beta}(V) \stackrel{\mathcal F}{\longrightarrow} C^\beta(S^{m-1},\Real) \stackrel{\mathcal I}{\longrightarrow} \Real,
\end{equation*}
where
\begin{equation*}
	\mathcal F(u)= \abs{\triangle_{S^{m-1}} \Pi(\varphi+u)}^2 +  (2m-8) \abs{\nabla_{S^{m-1}} \Pi(\varphi+u)}^2
\end{equation*}
and
\begin{equation*}
	\mathcal I(h)=\int_{S^{m-1}} h d\theta.
\end{equation*}
Recall that $\Pi$ is the nearest-point-projection of $N$ and the discussion works only for $u$ with small $C^0$ norm.

We claim that there exists an analytic map $\tilde{F}_C$ from $C^{2,\beta}(V)\otimes \mathbb C$ to $\mathbb C$ with $\tilde{F}$ as its real part.

The proof of the claim is the combination of following facts.
\begin{enumerate}
	\item[(F1)] The $\triangle_{S^{m-1}}$ from $C^{2,\beta}(V)$ to $C^{0,\beta}(V)$, $\nabla_{S^{m-1}}$ from $C^{2,\beta}(V)$ to $C^{1,\beta}(V)$ and $\mathcal I$ are bounded linear maps. Their complexifications, obtained by linear extension, are naturally bounded linear map and hence analytic.

	\item[(F2)] Let $\mathcal F_1$ be the map from $C^{0,\beta}(S^{m-1},\Real^p)$ to $C^{0,\beta}(S^{m-1},\Real)$ given by $u\mapsto \abs{u}^2$. Its complexification, $\mathcal F_{C,1}$ is given by
		\begin{equation*}
			\mathcal F_{C,1}(u+iv)= (u+iv)\cdot (u+iv).
		\end{equation*}
		It is analytic.

	\item[(F3)] If (as assumed in Theorem \ref{thm:main}) $\Pi(\varphi+\cdot)$ is an analytic map from $B_r(0)\subset \Real^p$ to $\Real^p$, then the map
		\begin{equation*}
			u\mapsto \Pi(\varphi+u)
		\end{equation*}
		has an analytic extension from $C^{2,\beta}(S^{m-1},\mathbb C^p)$ to itself. To see this, one first expands $\Pi(\varphi+u)$ into converging power series of $u$ and then replace $u$ by $u+iv$. It is then an exercise to check that the map thus obtained are analytic in the sense of Definition \ref{defn:analytic}.
\end{enumerate}

For the complexification of $\mathcal N$, it suffices to consider $\mathcal M_{\tilde{F}}(u)$. For $u$ and $v$ in $C^4(V)$, setting $\psi=\Pi(\varphi+u)$, we compute
\begin{eqnarray*}
	&& \frac{d}{dt}|_{t=0} \tilde{F}(u+tv)\\
	&=& \frac{d}{dt}|_{t=0} \int_{S^{m-1}} \abs{\triangle_{S^{m-1}} \Pi(\varphi+u+tv)}^2 + (2m-8) \abs{\nabla_{S^{m-1}}\Pi(\varphi+u+tv)}^2 d\theta \\
	&=& 2\int_{S^{m-1}} \triangle_{S^{m-1}}\psi \triangle_{S^{m-1}} D\Pi_{\varphi+u} v + (2m-8) \nabla_{S^{m-1}}\psi \nabla_{S^{m-1}} D\Pi_{\varphi+u} v d\theta \\
	&=& 2\int_{S^{m-1}} \left( \triangle_{S^{m-1}}^2 \psi - (2m-8) \triangle_{S^{m-1}}\psi \right)D\Pi_{\varphi+u} v d\theta \\
	&=& 2\int_{S^{m-1}} P_{\psi}\left( \triangle_{S^{m-1}}^2 \psi - (2m-8) \triangle_{S^{m-1}}\psi \right)  v d\theta.
\end{eqnarray*}
Here in the last line above, we used the fact that $D\Pi_{\varphi+u} v$ is nothing but the orthogonal projection from $\Real^p$ onto $T_{\psi} N$, which we denote by $P_{\psi}$.

Similar to the (bi)harmonic map case, $P_\psi(\triangle^2_{S^{m-1}} \psi - (2m-8) \triangle_{S^{m-1}} \psi)$ is the Euler-Lagrange operator of $F(\psi)$, denoted by $\mathcal M_F(\psi)$. For each $\omega\in S^{m-1}$, $\mathcal M_F(\psi)(\omega)$ lies in $T_\psi N\subset \Real^p$, while $v(\omega)$ is in $T_\varphi N$. Therefore,
\begin{equation}\label{eqn:MM}
	\mathcal M_{\tilde{F}}(u)=P_\varphi \mathcal M_F(\psi)
\end{equation}
where $\psi=\Pi(\varphi+u)$.

Since the projection $P_\varphi$ is a linear map that does not depend on $u$, the complexification of $\mathcal M_{\tilde{F}}(u)$ is reduced to that of $\mathcal M_{F}(\Pi(\varphi+u))$, which we regard as the composition of the following
\begin{enumerate}
	\item[(M1)] the map $\Pi(\varphi+\cdot)$ from $C^{4,\beta}(V)$ to $C^{4,\beta}(S^{m-1},\Real^p)$, which has been discussed in (F3) above;
	\item[(M2)] the map $\triangle_{S^{m-1}}^2 \psi -(2m-8)\triangle_{S^{m-1}} \psi$, from $C^{4,\beta}(S^{m-1},\Real^p)$ to $C^{0,\beta}(S^{m-1},\Real^p)$, which has been discussed in (F1) above;
	\item[(M3)] the projection $P_\psi$, is a $p$ by $p$ matrix that depends analytically on $\psi$, since $N$ is an analytic submanifold. Keeping in mind that $\psi=\Pi(\varphi + u)$ is known (see (M1) above) to be analytic map in $u$, the complexification of $P_\psi$ is given by expanding the analytic (matrix-valued) map $P_{\psi}=P_{\Pi(\varphi+u)}$ as a converging power series of $u$ and then replacing $u$ by $u+iv$ as in (F3).
\end{enumerate}

\subsection{Properties of the complexification}\label{subsec:properties}
Let's denote the complexification of $\mathcal M_{\tilde{F}}$ by $\mathcal M_{\tilde{F},C}$. In this section, we study the ellipticity of $\mathcal M_{\tilde{F},C}$ and the self-adjointness of its linearization at $0$. Please notice that although the ellipticity of $\mathcal M_{\tilde{F}}$ is quite natural, the ellipticity of $\mathcal M_{\tilde{F},C}$ as an operator between the complexified Banach spaces is not true in general. Fortunately, we have

\begin{lem}
	\label{lem:elliptic} The linearizations of both $\mathcal M_{\tilde{F}}$ and $\mathcal M_{\tilde{F},C}$ at $u=0$ are elliptic.
\end{lem}
\begin{rem}
	In fact, as the following proof shows, $\mathcal M_{\tilde{F}}$ is elliptic for small $u$ such that it is defined and $\mathcal M_{\tilde{F},C}$ is elliptic at $u+iv\in C^{4,\beta}(V)\otimes \mathbb C$ if $v=0$.
\end{rem}

\begin{proof}
Neglecting the lower order part, it suffices to compute the linearization of
\begin{equation*}
	P_\varphi P_\psi \triangle_{S^{m-1}}^2 \Pi(\varphi+u)
\end{equation*}
where $\psi=\Pi(\varphi+u)$. If we do the computation at $u\in C^{4,\beta}(V)$ with infinitesimal increment $h$ and neglects all lower order terms, we get
\begin{equation}\label{eqn:linearization}
	P_\varphi P_\psi \triangle_{S^{m-1}}^2 h,
\end{equation}
whose symbol is for any $\xi\in T_\omega^*S^{m-1}$,
\begin{equation}\label{eqn:symbol}
	\xi\mapsto P_\varphi P_\psi \abs{\xi}^4 h.
\end{equation}
If $\xi$ is not zero, then this is clearly a linear isomorphism from the sections of $V$ onto itself, because $\psi$ is close to $\varphi$.

Now, for $\mathcal M_{\tilde{F},C}$, we denote the complexification of $P_\psi$ ($\Pi(\varphi+u)$) by $P_{\psi,C}$ ($\Pi_C(\varphi+u)$ respectively). Although we do not know any exact formula for them, it suffices for us to note that when computing \eqref{eqn:linearization}, (1) the contribution of $\Pi_C$ goes to the lower order terms and does not matter; (2) since we have assumed that $u\in C^{4,\beta}(V)$, by the definition of complexification, $P_{\psi,C}=P_{\psi}$. Therefore, we get the same symbol as in \eqref{eqn:symbol}, which is now an isomorphism from the sections of complexified-$V$ onto itself.
\end{proof}

If we denote the linearizations of $\mathcal M_{\tilde{F}}$ and $\mathcal M_{\tilde{F},C}$ at $u=0$ by $\mathcal L_{\tilde{F}}$ and ${\mathcal L}_{\tilde{F},C}$, then
\begin{lem}
	\label{lem:L} For any $u,v\in C^{4,\beta}(V)$,
	\begin{equation}\label{eqn:ll}
		{\mathcal L}_{\tilde{F},C}(u+iv) = \mathcal L_{\tilde{F}}(u)+i \mathcal L_{\tilde{F}}(v).
	\end{equation}
	In particular, ${\mathcal L}_{\tilde{F},C}$ is an elliptic and self-adjoint operator from $C^{4,\beta}(V)\otimes \mathbb C$ to $C^{0,\beta}(V)\otimes \mathbb C$.
\end{lem}
\begin{proof}
	By definition, ${\mathcal L}_{\tilde{F},C}(u)=\frac{d}{dt}|_{t=0} \mathcal M_{\tilde{F},C}(tu)=\frac{d}{dt}|_{t=0}\mathcal M_{\tilde{F}}(tu)=\mathcal L_{\tilde{F}}(u)$. Hence, it suffices to show
	\begin{equation*}
		\tilde{\mathcal L}_{\tilde{F},C}(iv) = i \mathcal L_{\tilde{F}}(v).
	\end{equation*}
	Since $\mathcal M_{\tilde{F}}$ is a composition of $P_\varphi$, $P_\psi$, $\triangle_{S^{m-1}}$, $\nabla_{S^{m-1}}$ and $\Pi(\varphi+\cdot)$, it suffices to show that \eqref{eqn:ll} holds for (the linearization of) each one of them. This is trivial for $P_\varphi$, $\triangle_{S^{m-1}}$ and $\nabla_{S^{m-1}}$, because they are linear operators and \eqref{eqn:ll} is exactly how their complexification is defined.

	For $\Pi(\varphi+\cdot)$, we recall that
	\begin{equation*}
		\Pi_C (\varphi + (u+iv)) = \sum_{k} a_k (u+iv)^k
	\end{equation*}
	and the series converges for small $u$ and $v$. \eqref{eqn:ll} then follows from direct computation.  The same argument works for $P_\psi$.

	The self-adjointness of $\mathcal L_{\tilde{F}}$ follows from expanding the following identity
	\begin{equation*}
		\frac{d}{ds}|_{s=0} \frac{d}{dt}|_{t=0} \tilde{F}(tu+sv)= \frac{d}{dt}|_{t=0} \frac{d}{ds}|_{s=0} \tilde{F}(tu+sv).
	\end{equation*}
	The self-adjointness of $\mathcal L_{\tilde{F},C}$ is then a consequence of \eqref{eqn:ll}.	
	
\end{proof}

Now, we state the result that motivates the discussion in this section.
\begin{lem}
	\label{lem:justify} For $f$ defined \eqref{eqn:f}, it is analytic function of $\xi$ in a neighborhood of $0\in \Real^l$.
\end{lem}
\begin{proof}
	Let $\mathcal N_C$ be the complexification of $\mathcal N$ defined in Section \ref{subsec:complexification}. Its linearization at $u=0$	is given by
	\begin{equation*}
		P_K + \tilde{\mathcal L}.
	\end{equation*}
	By the results above, this $\tilde{\mathcal L}$ is elliptic and self-adjoint with trivial kernel. Hence, the inverse function theorem (Theorem \ref{thm:IFT}) gives an inverse map $\Psi_C$, which is analytic, from a neighborhood of $0$ in $C^{0,\beta}(V)\otimes \mathbb C$ to a neighborhood of $0$ in $C^{4,\beta}(V)\otimes \mathbb C$. If $\tilde{F}_C$ is the complexification of $\tilde{F}$ given in Section \ref{subsec:complexification}, then $f$ in \eqref{eqn:f} is the restriction (to the real part of $(z_1,\cdots,z_l)$) of
	\begin{equation*}
		f_C(z_1,\cdots,z_l):=\tilde{F}_C (\Psi_C (\sum_{i=1}^l z_i \varphi_i)),
	\end{equation*}
	which is analytic in (a neighborhood of $0$ in) $\mathbb C^l$
\end{proof}

\section{Proof of Lemma \ref{Lemma1}}\label{sec:Lemma1}
For some $\delta>0$ to be determined, we will take $W=\set{f|\, \norm{f}_{C^{\beta}(V)}<\delta}$. For any $f_1,f_2$ in $W$, we have
\begin{eqnarray*}
	&&\Psi(f_1)-\Psi(f_2) \\
	&=& \int_0^1 \frac{d}{dt} \Psi (tf_1+(1-t)f_2) dt \\
	&=& \int_0^1 D\Psi|_{tf_1+(1-t)f_2} (f_1-f_2) dt.
\end{eqnarray*}
Hence, it suffices to show that for any $f\in W$, the linearization of $\Psi$ at $f$, $D\Psi|_f$ is uniformly bounded linear operator from $L^2(V)$ to $W^{4,2}(V)$. More precisely, we need to find $\delta>0$ and $C>0$ such that
\begin{equation*}
	\sup_{f\in W} \norm{D\Psi|_f}_{L(L^2(V),W^{4,2}(V))} \leq C.
\end{equation*}
Here $\norm{\cdot}_{L(L^2(V),W^{4,2}(V))}$ is the norm of linear operators.

Since $\Psi$ is the inverse of $\mathcal N$, it suffices to show that there exist $\delta'>0$ and $C>0$ such that if $W'=\set{u\in C^{4,\beta}(V)|\, \norm{u}_{C^{4,\beta}(V)}<\delta'}$,
\begin{equation}\label{eqn:dn}
	\inf_{u\in W'} \norm{D\mathcal N|_u}_{L(W^{4,2},L^2)} \geq C>0.
\end{equation}

The proof of \eqref{eqn:dn} consists of two steps. First, we show that
\begin{equation}\label{eqn:dn0}
	\norm{D\mathcal N|_0}_{L(W^{4,2},L^2)} \geq C>0.
\end{equation}
Recall that $\mathcal N=P_K + \mathcal M_{\tilde{F}}$, where $K$ is the kernel of $D\mathcal M_{\tilde{F}}|_0=\mathcal L_{\tilde{F}}$. For any $h\in W^{4,2}(V)$, we denote $h-P_K h$ by $h^\perp$. Since $\mathcal L_{\tilde{F}}$ is an elliptic operator with trivial kernel in the compliment space of $K$, there is a constant depending only on $\varphi$ such that
\begin{equation}\label{eqn:hperp}
	\norm{h^\perp}_{W^{4,2}}\leq C_0 \norm{\mathcal L_{\tilde{F}} h^\perp}_{L^2}.
\end{equation}
Since $K$ is a finite-dimensional space, there is $C_1>0$ such that
\begin{equation}\label{eqn:pkh}
	\norm{P_K h}_{W^{4,2}}\leq C_1 \norm{P_K h}_{L^2}.
\end{equation}
Combining \eqref{eqn:hperp} and \eqref{eqn:pkh} and noticing that the image of $\mathcal L_{\tilde{F}}$ is normal to $K$ in $L^2$, we get $C_2>0$ such that
\begin{equation}\label{eqn:goodh}
	\norm{h}_{W^{4,2}}\leq C_2 \norm{D\mathcal N|_0 h}_{L^2},
\end{equation}
which implies \eqref{eqn:dn0}.

The second step is to show that for $u\in W'$,
\begin{equation}\label{eqn:2step}
	\norm{(D\mathcal N|_u-D\mathcal N|_0)h}_{L^2}\leq C(\delta') \norm{h}_{W^{4,2}}
\end{equation}
for some $C(\delta')$ satisfying $\lim_{\delta'\to 0} C(\delta')=0$. Before the proof of \eqref{eqn:2step}, we notice that if $\delta'$ is small, \eqref{eqn:2step} and \eqref{eqn:goodh} imply that
\begin{equation*}
	\norm{h}_{W^{4,2}}\leq C \norm{D\mathcal N|_u h}_{L^2},
\end{equation*}
which finishes the proof of \eqref{eqn:dn} and hence the proof of Lemma \ref{Lemma1}.

For \eqref{eqn:2step}, we notice that the contribution of $P_K h$ cancels out and it suffices to bound
\begin{equation}\label{eqn:final}
	\norm{(D\mathcal M_{\tilde{F}}|_u-D\mathcal M_{\tilde{F}}|_0)h}_{L^2}.
\end{equation}
Recalling that $\mathcal M_{\tilde{F}}(u)=P_\varphi P_\psi (\triangle_{S^{m-1}}^2 \psi -(2m-8)\triangle_{S^{m-1}}\psi)$ with $\psi=\Pi(\varphi+u)$, we get
\begin{eqnarray}\label{eqn:dmh1}
	D\mathcal M_{\tilde{F}}|_uh
	&=& P_\varphi (DP)_\psi \left( \triangle_{S^{m-1}}^2 \psi - (2m-8) \triangle_{S^{m-1}}\psi \right) h \\ \nonumber
	&&+ P_\varphi P_\psi \left( \triangle_{S^{m-1}}^2 h -(2m-8) \triangle_{S^{m-1}}h \right)
\end{eqnarray}
and
\begin{eqnarray}\label{eqn:dmh2}
	D\mathcal M_{\tilde{F}}|_0h
	&=& P_\varphi (DP)_\varphi \left( \triangle_{S^{m-1}}^2 \varphi - (2m-8) \triangle_{S^{m-1}}\varphi \right) h \\ \nonumber
	&&+ P_\varphi  \left( \triangle_{S^{m-1}}^2 h -(2m-8) \triangle_{S^{m-1}}h \right)
\end{eqnarray}
Notice that  \eqref{eqn:dmh1} and \eqref{eqn:dmh2} are fourth order linear operator of $h$ and if we subtract them, the difference of the corresponding coefficients are bounded by using
\begin{equation*}
	\norm{\psi-\varphi}_{C^{4,\beta}}\leq C \norm{u}_{C^{4,\beta}(V)}.
\end{equation*}
\bibliographystyle{plain}
\bibliography{foo}

\end{document}